\documentclass[onecolumn,12pt,lettersize]{IEEEtran}
\usepackage[top=1 in, bottom=1 in, left=0.75 in, right=0.75 in]{geometry}

\usepackage[cmex10]{amsmath} 
\usepackage{amssymb}
\usepackage{tabularx}
\usepackage[sc,osf,noBBpl]{mathpazo} 
\usepackage{cite} 
\usepackage[pdftex]{graphicx} 
\usepackage{subcaption}
\usepackage{algorithm}
\usepackage[noend]{algorithmic}
\usepackage{amssymb} 
\usepackage{graphicx}
\usepackage{epstopdf}

\usepackage{stfloats}
\usepackage{tikz}

\usepackage{amsthm} 
\newtheorem{theorem}{Theorem}
\newtheorem{lemma}{Lemma}

\newtheorem{cor}{Corollary}

\begin{document}

\title{A Simple Steady-State Analysis of Load Balancing Algorithms in the Sub-Halfin-Whitt Regime}

\author{\IEEEauthorblockN{Xin Liu and Lei Ying}\\
\IEEEauthorblockA{School of Electrical, Computer and Energy Engineering,\\
Arizona State University\\
Email: \{xliu272, lei.ying.2\}@asu.edu}
}

\maketitle

\section{abstract}

This paper studies a class of load balancing algorithms  for many-server ($N$ servers) systems assuming finite buffer with size $b-1$ (i.e. a server can have at most one job in service and $b-1$ jobs in queue). We focus on steady-state performance of load balancing algorithms in the heavy traffic regime such that the load of system is $\lambda = 1 - N^{-\alpha}$ for $0<\alpha<0.5,$ which we call sub-Halfin-Whitt regime ($\alpha=0.5$ is the so-called the Halfin-Whitt regime). We establish a sufficient condition under which the probability that an incoming job is routed to an idle server is one asymptotically. The class of load balancing algorithms that satisfy the condition includes join-the-shortest-queue (JSQ), idle-one-first (I1F), join-the-idle-queue (JIQ), and power-of-$d$-choices (Po$d$) with $d=N^\alpha\log N$. The proof of the main result is based on the framework of Stein's method. A key contribution is to use a simple generator approximation based on state space collapse.

\section{Introduction}
This paper studies the steady-state performance of load balancing algorithms in many-server systems. We consider a system with $N$ identical servers with {\em buffer size $b-1$ such that $b=o\left(\sqrt{\log N}\right),$ in other words, each server can hold at most $b$ jobs, one job in service and $b-1$ jobs in buffer.} We assume jobs arrive according to a Poisson process with rate $\lambda N,$ where $\lambda=1-N^{-\alpha}$ for $0<\alpha<0.5,$ and have exponential service times with mean one. When a job arrives, the load balancer immediately routes the job to one of the servers. If the server's buffer is full, the job is discarded. We study a class of load balancing algorithms, which includes join-the-shortest-queue (JSQ), idle-one-first (I1F) \cite{GuptaWalton_17}, join-the-idle-queue (JIQ) \cite{LuXieKli_11,Sto_15} and power-of-$d$-choices (Po$d$) with $d=N^\alpha\log N$ \cite{Mit_96,VveDobKar_96}, and  establish an upper bound on the mean queue length. From the queue-length bound, we further show that under JSQ, I1F, and Po$d$ with $d=N^\alpha \log N,$ the probability that a job is routed to a non-idle server and the expected waiting time per job are both $O\left(\frac{\log N}{\sqrt{N}}\right),$ which means only $O\left(\frac{\log N}{\sqrt{N}}\right)$ fraction of jobs experience non-zero waiting or are discarded. For JIQ, we show that the probability of waiting is $O\left(\frac{b}{N^{0.5-\alpha}\log N}\right).$

\subsection{Related Work and Our Contributions}

Performance analysis of many-server systems is one of the most fundamental and widely-studied problems in queueing theory. The stationary distribution of the classic $M/M/N$ system (or called Erlang-C model) is one of the earliest subjects of queueing theory. For systems with distributed queues where each server maintains a separate queue, it is well known that the join-the-shortest-queue (JSQ) algorithm is delay optimal \cite{Win_77,Web_78} under fairly general conditions. However, the exact stationary distribution of many-server systems under JSQ remains to be an open problem.  A recent breakthrough in this area is \cite{EscGam_18}, which shows that in the Halfin-Whitt regime ($\alpha=0.5$), the diffusion-scaled process converges to a two-dimensional diffusion limit, from which it can be shown that most servers have one job in service and $O(\sqrt{N})$ servers have two jobs (one in service and one in buffer). This seminal work has led to several significant developments: (i)  \cite{Bra_18} proved that the stationary distribution indeed converges to the  stationary distribution of the two-dimensional diffusion limit based on Stein's method; and (ii) via stochastic coupling, \cite{MukBorvan_16} showed that the diffusion limit of Po$d$ converges to that of JSQ in the Halfin-Whitt regime at the process level (over finite time) when $d=\Theta(\sqrt{N}\log N);$ and (iii) when $\alpha<1/6,$ \cite{LiuYin_18} proved that the waiting probability of a job is asymptotically zero with $d=\omega\left(\frac{1}{1-\lambda}\right)$ at the steady-state based on Stein's method. Interested readers can find a comprehensive survey of recent results in \cite{VanBorVan_17}. 

Let $S_i$ denote the fraction of servers with at least $i$ jobs {\em at steady state}. In this paper, we prove that
\begin{align*}
E\left[\max\left\{\sum_{i=1}^b S_i-\lambda -\frac{k\log N}{\sqrt{N}},0\right\}\right]\leq \frac{29b}{\sqrt{N}\log N}, ~\text{with}~ k=1+\frac{1}{2(b-1)},
\end{align*} for a class of load balancing algorithms that route an incoming job to an idle server with probability at least $1-\frac{1}{\sqrt{N}}$ when $S_1\leq \lambda+\frac{k\log N}{\sqrt{N}}.$  This result implies that  (i) $$E\left[\sum_{i=1}^b {S}_i\right]\leq \lambda +\frac{k\log N}{\sqrt{N}}+\frac{29b}{\sqrt{N}\log N},$$ i.e, the average queue length per server exceeds $\lambda$ by at most $O\left(\frac{\log N}{\sqrt{N}}\right);$  and (ii)  under JSQ,  I1F, JIQ and Po$d$ ($d=N^{\alpha}\log N$), the probability that an incoming job is routed to a non-idle server is asymptotically zero.

From the best of our knowledge, there are only a few papers that deal with the steady-state analysis of many-server systems with distributed queues \cite{Bra_18,BanMuk_18,LiuYin_18}. \cite{Bra_18,BanMuk_18} analyze the steady-state distribution of JSQ in the Halfin-Whitt regime and \cite{LiuYin_18} studies the Po$d$ with $\alpha<1/6.$ This paper complements \cite{Bra_18,BanMuk_18,LiuYin_18}, as it applies to a class of load balancing algorithms and to any sub-Halfin-Whitt regime.

Similar to \cite{Bra_18,LiuYin_18}, the result of this paper is proved using the mean-field approximation (fluid-limit approximation) based on Stein's method. The execution of Stein's method in this paper, however, is quite different from \cite{Bra_18,LiuYin_18}. In our proof, a simple mean-field model (fluid-limit) model $\sum_{i=1}^b \dot{S}_i =-\frac{\log N}{\sqrt{N}}$  is used to partially approximate the evolution of the stochastic system when the system is away from the mean-field equilibrium. This is because in this paper, we are interested in bounding $$E\left[\max\left\{\sum_{i=1}^b {S}_i-\lambda -\frac{k\log N}{\sqrt{N}},0\right\}\right],$$ i.e. when $\sum_{i=1}^b {S}_i\geq \lambda +\frac{k\log N}{\sqrt{N}}>\lambda.$ Note that this simple mean-field model is not even accurate when $\sum_{i=1}^b {S}_i\geq \lambda +\frac{k\log N}{\sqrt{N}}.$ However, using state-space collapse (SSC) approach based on the tail bound in \cite{BerGamTsi_01}, we show that the generator difference is small. In the literature, SSC has been used to show that the approximation error of using a low-dimensional system is order-wise smaller than the queue length (or some function of the queue length). Instead in this paper, we show that the error is a fraction of $E\left[\max\left\{\sum_{i=1}^b {S}_i-\lambda-\frac{k\log N}{\sqrt{N}}\right\},0\right]$, but not negligible, with a high probability. We then deal with this error by subtracting it from  $E\left[\max\left\{\sum_{i=1}^b {S}_i-\lambda -\frac{k\log N}{\sqrt{N}},0\right\}\right]$ without bounding it explicitly. Furthermore, SSC is proved only in the regime $\sum_{i=1}^b {S}_i\geq \lambda +\frac{k\log N}{\sqrt{N}},$ which turns out to be sufficient and easy to prove. Pioneered in \cite{Sto_15_2} (called drift-based-fluid-limits (DFL) method) for fluid-limit analysis and in \cite{BraDaiFen_15,BraDai_17} for steady-state diffusion approximation, the power of Stein's method for steady-state approximations has been recognized in a number of recent papers \cite{Sto_15_2,BraDaiFen_15,Yin_16,BraDai_17,Yin_17,Gas_17,Gas_18,Bra_18}.  All proofs in this paper are elementary. Therefore, this paper is another an example that demonstrates the power of Stein's method for analyzing complex queueing systems with elementary probability methods.

\section{Model and Main Results}
Consider a many-server system with $N$ homogeneous servers, where job arrival follows a Poisson process with rate $\lambda N$ and service times are i.i.d. exponential random variables with rate one. We consider the sub-Halfin-Whitt regime such that $\lambda=1-N^{-\alpha}$ for some $0<\alpha<0.5.$  As shown in Figure \ref{model-pod}, each server maintains a separate queue and we assume buffer size $b-1$ (i.e., each server can have one job in service and $b-1$ jobs in queue).

\begin{figure}[!htbp]
  \centering
  \includegraphics[width=3.3in]{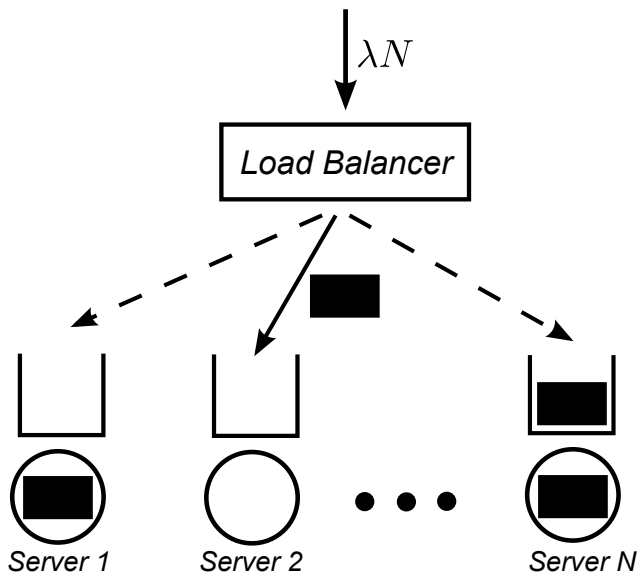}
  \caption{Load Balancing in Many-Server Systems.}
  \label{model-pod}
\end{figure}

We study a class of load balancing algorithms which route each incoming job to a server upon its arrival. Denote by $S_i(t)$ the fraction of servers with queue length at least $i$ at time $t.$ Under the finite buffer assumption with buffer size $b$, $S_i = 0, \forall i \geq b+1.$  Define $\mathcal S$ to be $$\mathcal S = \{ s ~|~ 1\geq s_1\geq \cdots \geq s_{b}\geq  0\},$$ and $S(t) = [S_1(t), S_2(t), \cdots ,S_b(t)].$ We consider load balancing algorithms such that $S(t) \in \mathcal S$ is a continuous-time Markov chain (CTMC) and has a unique stationary distribution, denoted by $S,$ for any $\lambda$. Note $\lambda,$ $S(t)$ and $S$ all depend on $N,$ the number of servers in the system. Let $A_1(S)$ denote the probability that an incoming job is routed to a busy server when the state of the system is $S.$ Our main result of this paper is the following theorem.
\begin{theorem} \label{Thm:main}
Assume $\lambda=1-N^{-\alpha},$ $0<\alpha<0.5,$ and $b = o(\sqrt{\log N}).$ Under any load balancing algorithm such that $A_1(S)\leq \frac{1}{\sqrt{N}}$ when $S_1\leq \lambda+\frac{k\log N}{\sqrt{N}}$ with $k = 1+\frac{1}{2(b-1)},$ the following bound holds when $N$ is sufficiently large:
\begin{align*}
E\left[\max\left\{\sum_{i=1}^{b} S_i-\lambda -\frac{k\log N}{\sqrt{N}},0\right\}\right]\leq \frac{29b}{\sqrt{N}\log N}.
\end{align*}
\end{theorem}

Note that the condition $A_1(S)\leq \frac{1}{\sqrt{N}}$ when $S_1\leq \lambda+\frac{k\log N}{\sqrt{N}}$ implies that an incoming job should be routed to an idle server with probability at least $1-\frac{1}{\sqrt{N}}$ when at least $\frac{1}{N^{\alpha}}-\frac{k\log N}{\sqrt{N}}$ fraction of servers are idle. There are several well-known policies that satisfy this condition.
\begin{itemize}
\item {\bf Join-the-Shortest-Queue (JSQ)}: JSQ routes an incoming job to the least loaded server in the system, so $A_1(S)=0$ when $S_1\leq \lambda+\frac{k\log N}{\sqrt{N}}.$

\item {\bf Idle-One-First (I1F)}: I1F routes an incoming job to an idle server if available and else to a server with one job if available. Otherwise, the job is routed to a randomly selected server. Therefore, $A_1(S)=0$ when $S_1\leq \lambda+\frac{k\log N}{\sqrt{N}}.$  

\item {\bf Join-the-Idle-Queue (JIQ)}: JIQ routes an incoming job to an idle server if possible and otherwise, routes the job to server chosen uniformly at random. Therefore, $A_1(S)=0$ when $S_1\leq \lambda+\frac{k\log N}{\sqrt{N}}.$

\item {\bf Power-of-$d$-Choices (Po$d$)}:  Po$d$ samples $d$ servers uniformly at random and dispatches the job to the least loaded server among the $d$ servers. Ties are broken uniformly at random. When $d=N^\alpha\log N,$ $A_1(S)\leq \frac{1}{\sqrt{N}}$ when $S_1\leq \lambda+\frac{k\log N}{\sqrt{N}}.$
\end{itemize}

A direct consequence of Theorem \ref{Thm:main} is asymptotic zero waiting.  Let $\mathcal W_N$ denote the event that an incoming job is routed to a busy server in a system with $N$ servers, and $p_{\mathcal W_N}$ denote the probability of this event at the steady-state. Let $\mathcal B_N$ denote the event that an incoming job is blocked (discarded) and $p_{\mathcal B_N}$ denote the probability of this event at the steady-state. Furthermore, let $W_N$ denote the waiting time of a job (when the job is not dropped). We have the following results based on the main theorem.

\begin{cor} \label{Thm:zerodelay}
Assume $\lambda=1-N^{-\alpha},$ $0<\alpha<0.5,$ and $b=o\left(\sqrt{\log N}\right).$ For sufficiently large $N,$ we have
\begin{itemize}
\item Under JSQ, IF1, and Po$d$ with $d=N^\alpha\log N,$ $$E\left[W_N\right]\leq \frac{3\log N}{\sqrt{N}}, \quad\hbox{and}\quad p_{\mathcal {\cal W}_N}\leq \frac{4\log N}{\sqrt{N}}.$$

\item Under JIQ, $$p_{\mathcal {\cal W}_N}\leq \frac{30b}{N^{0.5-\alpha}\log N}.$$

\end{itemize}
\end{cor} The proof of this lemma is a simple application of the Markov inequality, which can be found in the Section \ref{sec:0-delay}.

We next provide an overview of the proof of our main theorem. The details are presented in the section \ref{sec:proof}. The proof is based on Stein's method. As modularized in \cite{BraDai_17}, this approach includes three key ingredients: generator approximation, gradient bounds and state space collapse (SSC).

Define $e_i$ to be a $b$-dimensional vector such that the $i$th entry is $1/N$ and all other entries are zero. Furthermore, define $A_i(S)$ to be the probability that an incoming job is routed to a server with at least $i$ jobs. For convenience, define $A_0(S)=1$ and $A_{b+1}(S)=B(S),$ where $B(S)$ is the probability that an incoming job is discarded. Let $G$ be the generator of CTMC $S(t).$ Given function $g: \mathcal S \to R,$ we have
\begin{align}
G g(S) = &\sum_{i=1}^{b}\lambda N (A_{i-1}(S)-A_i(S)) (g(S + e_i) - g(S)) +  N (S_i - S_{i+1}) (g(S - e_i)-g(S)) \nonumber
\end{align}  For a bounded function $g: \mathcal S \to R,$
$$E[ G g(S) ] = 0.$$

Following the framework of Stein's method, the first step of our proof is generator approximation. We propose a simple, almost trivial, generator $L$ such that
\begin{align*}
L g(s) = g'(s) \left(-\frac{\log N}{\sqrt{N}}\right),
\end{align*} and assume $g(s)$ is the solution of the following Stein's equation (also called Poisson equation):
\begin{align*}
L g(s) = g'(s) \left(-\frac{\log N}{\sqrt{N}}\right)=h(s).
\end{align*}
Following Stein's method, we bound $E[h(s)]$ by studying generator difference between $L$ and $G:$
\begin{align*}
E[h(S)] =& E[L g(S) - Gg(S)] = E[g'(S)  \left(-\frac{\log N}{\sqrt{N}}\right) - Gg(S)] \\
        =& E\left[g'(S)\left(\lambda B(S)-\lambda-\frac{\log N}{\sqrt{N}}+S_1\right)+\frac{c}{N}g''(S)\right]
\end{align*} for some constant $c>0.$
The second term  can be bounded by using the gradient bound on $g''(s).$ We will see that $g''(s)$ has a very simple form and is trivial to calculate. The first term can be bounded by bounding the generator difference, which is established based on SSC in the regime $\sum_{i=1}^{b} S_i\geq \lambda+\frac{k\log N}{\sqrt{N}}.$

\section{Proof of Theorem \ref{Thm:main}} \label{sec:proof}
In this section, we present the proof of our main theorem, which is organized along the three key ingredients.

\subsection{Generator Approximation based on Stein's Method}

Define  function $h(s) = \max\left\{s- \lambda -\frac{k\log N}{\sqrt{N}},0\right\}.$  This function is motivated by \cite{Sto_15_2}, which uses a smooth function to approximate L1 distance $| \cdot |$ and proves the tightness of the ``$N$-system" in the Halfin-Whitt regime. We choose $h(s) = \max\left\{s- \lambda -\frac{k\log N}{\sqrt{N}},0\right\}$ to study the probability that  $\sum_{i=1}^b S_i \geq \lambda +\frac{k\log N}{\sqrt{N}}.$

We use a simple generator $L$ such that $L g(s)=g'(s)\left(-\frac{\log N}{\sqrt{N}}\right),$  and consider function $g$ such that
\begin{align*}
L g(s) = g'(s) \left(-\frac{\log N}{\sqrt{N}}\right)= \max\left\{s- \lambda -\frac{k\log N}{\sqrt{N}},0\right\}, ~ \text{with}~ g(0) = 0.
\end{align*}

From the definition of $g$ function, note that for any $s\leq  \lambda +\frac{k\log N}{\sqrt{N}},$
\begin{equation*}
g(s)=0,
\end{equation*} because
\begin{equation*}
g'(s)=0,
\end{equation*} for any $s\leq  \lambda +\frac{k\log N}{\sqrt{N}}.$ It implies for any $s \leq \lambda +\frac{k\log N}{\sqrt{N}}-\frac{1}{N},$
$$g(s+\frac{1}{N})=g(s)=g(s-\frac{1}{N})=0.$$

Leting $s= \sum_{i=1}^{b} S_i$, we have
\begin{align*}
&E\left[h\left(\sum_{i=1}^{b} S_i\right)\right] \\
=&E\left[g'\left(\sum_{i=1}^{b} S_i\right)\left(-\frac{\log N}{\sqrt{N}}\right)\right] - E\left[ G g\left(\sum_{i=1}^{b} S_i\right) \right] \\
=&E\left[g'\left(\sum_{i=1}^{b} S_i\right)\left(-\frac{\log N}{\sqrt{N}}\right)\right.\\
&\left.-N\lambda(1-B(S))\left(g\left(\sum_{i=1}^{b} S_i+\frac{1}{N}\right)-g\left(\sum_{i=1}^{b} S_i\right)\right)-NS_1\left(g\left(\sum_{i=1}^{b} S_i-\frac{1}{N}\right)-g\left(\sum_{i=1}^{b} S_i\right)\right)\right].
\end{align*}

Note $E[h(\sum_{i=1}^{b} S_i)]=0,$ for any $\sum_{i=1}^{b} S_i<  \lambda +\frac{k\log N}{\sqrt{N}}- \frac{1}{N}.$ Therefore, by Taylor's expansion of the function $g$ in the interval $\left( \lambda +\frac{k\log N}{\sqrt{N}}  + \frac{1}{N}, ~\infty \right)$ and the mean-value theorem of that in $[ \lambda +\frac{k\log N}{\sqrt{N}} - \frac{1}{N}, \lambda +\frac{k\log N}{\sqrt{N}} + \frac{1}{N}],$ we have
\begin{align}
&E\left[\max\left\{\sum_{i=1}^{b} S_i- \lambda -\frac{k\log N}{\sqrt{N}},0\right\}\right] \nonumber \\
\leq &E\left[g'\left(\sum_{i=1}^{b} S_i\right)\left(\lambda B(S)- \lambda-\frac{\log N}{\sqrt{N}}+S_1\right)\mathbb{I}_{\sum_{i=1}^{b} S_i> \lambda +\frac{k\log N}{\sqrt{N}} +\frac{1}{N}}\right] +\frac{2}{N}\max_{\sum_{i=1}^{b} S_i >  \lambda +\frac{k\log N}{\sqrt{N}}  }|g''(S)| \nonumber\\
&+E\left[\left(g'\left(\sum_{i=1}^{b} S_i\right)\left(-\frac{\log N}{\sqrt{N}}\right)+|\lambda-\lambda B(S)|g'(\xi)+|S_1|g'(\tilde{\xi})\right)\mathbb{I}_{-\frac{1}{N}\leq \sum_{i=1}^{b} S_i -\lambda -\frac{k\log N}{\sqrt{N}}  \leq \frac{1}{N}}\right], \label{G-expansion}
\end{align} where $\xi\in\left(\sum_{i=1}^{b} S_i, \sum_{i=1}^{b} S_i+\frac{1}{N}\right)$ and $\tilde{\xi}\in\left(\sum_{i=1}^{b} S_i-\frac{1}{N}, \sum_{i=1}^{b} S_i\right).$

In the following, we study $g'$ and $g''$ to bound the last two terms in the inequality \eqref{G-expansion}.

\subsection{Gradient Bounds}
We summarize bounds on $g'$ and $g''$ in the following two lemmas.

\begin{lemma}\label{lemma:g'}
Given  $s\in\left[ \lambda +\frac{k\log N}{\sqrt{N}}  -\frac{2}{N}, \lambda +\frac{k\log N}{\sqrt{N}}  +\frac{2}{N}\right],$ we have $$|g'(s)|\leq \frac{2}{\sqrt{N}\log N}.$$
\end{lemma}
\begin{proof}
Note that
$$g'(s)=\frac{\max\left\{s-\lambda -\frac{k\log N}{\sqrt{N}}  ,0\right\}}{-\frac{\log N}{\sqrt{N}}}.$$
Hence, we have
$$|g'(s)|\leq \frac{\frac{2}{N}}{\frac{\log N}{\sqrt{N}}}=\frac{2}{\sqrt{N}\log N}.$$
\end{proof}

\begin{lemma}\label{lemma:g''}
For $s>\lambda +\frac{k\log N}{\sqrt{N}},$ we have \begin{align*}
|g''(s)|\leq \frac{\sqrt{N}}{\log N}.
\end{align*}
\end{lemma}
\begin{proof}
For $s> \lambda +\frac{k\log N}{\sqrt{N}}$ we have
\begin{eqnarray*}
|g''(s)|=\left|\frac{\sqrt{N}}{\log N}\right|=\frac{\sqrt{N}}{\log{N}}.
\end{eqnarray*}
\end{proof}

\subsection{State Space Collapse (SSC)}

In this section, we consider the first term in \eqref{G-expansion}
\begin{align}
&g'\left(\sum_{i=1}^{b} S_i\right)\left(\lambda B(s)-\lambda-\frac{\log N}{\sqrt{N}} +S_1\right)\mathbb{I}_{\sum_{i=1}^{b} S_i> \lambda+\frac{k\log N}{\sqrt{N}} +\frac{1}{N}} \nonumber\\
=&-\frac{\sqrt{N}}{\log N}\left(\sum_{i=1}^{b} S_i-\lambda-\frac{k\log N}{\sqrt{N}} \right)\left(\lambda B(s)-\lambda- \frac{\log N}{\sqrt{N}} + S_1\right)\mathbb{I}_{\sum_{i=1}^{b} S_i> \lambda+\frac{k\log N}{\sqrt{N}} +\frac{1}{N}} \nonumber\\
= &\frac{\sqrt{N}}{\log N}\left(\sum_{i=1}^{b} S_i-\lambda-\frac{k\log N}{\sqrt{N}} \right)\left(\lambda+ \frac{\log N}{\sqrt{N}} - S_1-\lambda B(s)\right)\mathbb{I}_{\sum_{i=1}^{b} S_i> \lambda+\frac{k\log N}{\sqrt{N}} +\frac{1}{N}} \nonumber\\
\leq &\frac{\sqrt{N}}{\log N}\left(\sum_{i=1}^{b} S_i-\lambda-\frac{k\log N}{\sqrt{N}} \right)\left(\lambda+ \frac{\log N}{\sqrt{N}} - S_1\right)\mathbb{I}_{\sum_{i=1}^{b} S_i> \lambda+\frac{k\log N}{\sqrt{N}} +\frac{1}{N}}. \label{SSC}
\end{align}
where the last inequality holds because $\left(\sum_{i=1}^{b} S_i-\lambda-\frac{k\log N}{\sqrt{N}}\right)\mathbb{I}_{\sum_{i=1}^{b} S_i>\lambda+\frac{k\log N}{\sqrt{N}}+\frac{1}{N}}\geq 0.$

With minor abuse of notation, now let $s=[s_1,s_2,\cdots, s_b].$ We define a Lyapunov function
\begin{align}
V(s)=\min\left\{\sum_{i=2}^{b} s_i, \lambda+\frac{k\log N}{\sqrt{N}} -s_1 \right\}. \label{ly-func}
\end{align}

\begin{lemma}\label{DriftBound}
For sufficient large $N,$ we have
\begin{align*}
\triangledown V(s)\leq -\frac{1}{2(b-1)}\frac{\log N}{\sqrt{N}}+\frac{1}{\sqrt{N}},
\end{align*}
for any $s$ such that
\begin{align*}
V(s)\geq \frac{\log N}{\sqrt{N}}.
\end{align*}
\end{lemma}
\begin{proof}

For the Lyapunov function defined in (\ref{ly-func}), the Lyapunov drift is
\begin{align*}
&\triangledown V(s)=E\left[GV(S)|S=s\right]\\
= &\sum_{i=1}^{b}\lambda N (A_{i-1}(s)-A_i(s)) (V(s + e_i) - V(s)) +  N (s_i - s_{i+1}) (V(s - e_i)-V(s)). \nonumber
\end{align*}

Given $V(s)\geq \frac{\log N}{\sqrt{N}},$ we consider the following two cases.
\begin{itemize}
\item Case 1: Assume $\sum_{i=2}^{b} s_i\leq \lambda+\frac{k\log N}{\sqrt{N}}-s_1.$ Note that
\begin{align*} 
&V(s+e_1)\leq \sum_{i=2}^{b} s_i, ~ V(s-e_1) = \sum_{i=2}^{b} s_i,\\
&V(s+e_j)\leq \sum_{i=2}^{b} s_i+\frac{1}{N}, ~V(s-e_j) =  \sum_{i=2}^{b} s_i - \frac{1}{N}, ~ \forall ~ 2 \leq j \leq b.  
\end{align*}

Furthermore,
$V(s)=\sum_{i=2}^{b} s_i\geq \frac{\log N}{\sqrt{N}},$ which implies $s_2\geq \frac{1}{b-1}\frac{\log N}{\sqrt{N}}$ because $s_2\geq s_3\geq \cdots \geq s_b.$  Therefore, we have
\begin{align*}
\triangledown V(s)\leq\lambda(A_1(s)- B(s))-s_2\leq - \frac{1}{b-1}\frac{\log N}{\sqrt{N}} + \frac{1}{\sqrt{N}},
\end{align*} where the last inequality holds because $\sum_{i=1}^{b} s_i\leq \lambda+\frac{k\log N}{\sqrt{N}}$ implies that $s_1<\lambda+\frac{k\log N}{\sqrt{N}}$ which further implies that $A_1(s)\leq \frac{1}{\sqrt{N}}.$

\item Case 2: Assume $\sum_{i=2}^{b} s_i>\lambda+\frac{k\log N}{\sqrt{N}}-s_1.$ Note that 
\begin{align*} 
&V(s + e_1) = \lambda+\frac{k\log N}{\sqrt{N}}-s_1- \frac{1}{N}, ~ V(s - e_1)\leq \lambda+\frac{k\log N}{\sqrt{N}}-s_1+\frac{1}{N}, \\
&V(s + e_j) = \lambda+\frac{k\log N}{\sqrt{N}}-s_1, ~ V(s- e_j)\leq \lambda+\frac{k\log N}{\sqrt{N}}-s_1, ~\forall~ 2 \leq j \leq b. 
\end{align*}

In this case
$\sum_{i=2}^b s_i\geq V(s)=\lambda+\frac{k\log N}{\sqrt{N}}-s_1\geq \frac{\log N}{\sqrt{N}},$ which also implies $s_2\geq \frac{1}{b-1}\frac{\log N}{\sqrt{N}}.$
\begin{align*}
\triangledown V(s)\leq&-\lambda(1-A_1(s))+(s_1-s_2)\\
=& s_1-s_2-\lambda+\lambda A_1(s)\\
\leq& (k-1)\frac{\log N}{\sqrt{N}}-s_2+\lambda A_1(s)\\
\leq& \left(k-1-\frac{1}{b-1}\right)\frac{\log N}{\sqrt{N}}+ \frac{1}{\sqrt{N}} \\
=& -\frac{1}{2(b-1)}\frac{\log N}{\sqrt{N}}+ \frac{1}{\sqrt{N}}
\end{align*}
where the second last inequality holds because $s_1\leq \lambda+(k-1)\frac{\log N}{\sqrt{N}}$ and the last equality holds because $k = 1+\frac{1}{2(b-1)}.$
\end{itemize}

\end{proof}

Before moving forward, we present the following result from \cite{BerGamTsi_01}. The following version of the lemma is from \cite{WanMagSri_17}, but the result was proven in \cite{BerGamTsi_01}.
\begin{lemma} \label{TailBound}
Let $(X (t): t \geq 0)$ be a continuous-time Markov chain over a countable state space $X$. Suppose that it is irreducible, nonexplosive and positive-recurrent, and it converges in distribution to a random variable $\bar X.$ Consider a Lyapunov function $V: X \to R^{+}$ and define the drift of $V$ at a state $i \in \mathcal X$ as $$\Delta V(i) = \sum_{i' \in \mathcal X: i' \neq i} q_{i i'} (V(i') - V(i)),$$ where $q_{ii'}$ is the transition rate from $i$ to $i'.$ Suppose that the drift
satisfies the following conditions:

(i) There exists constants $\gamma > 0$ and $B > 0$ such that $\Delta V (i) \leq -\gamma$ for any $i \in X$ with $V(i) > B.$

(ii) $\nu_{\max} := \sup\limits_{i,i'\in \mathcal X: q_{i i'} >0} |V(i') - V(i)|< \infty.$

(iii) $\bar q := \sup\limits_{i \in \mathcal X} (-q_{ii}) < \infty.$

Then for any nonnegative integer $j$, we have
$$\Pr\left(V(\bar X) > B + 2 \nu_{\max} j\right) \leq \left(\frac{q_{\max}\nu_{\max}}{q_{\max}\nu_{\max} + \gamma}\right)^{j+1},$$ where $$q_{\max}=\sup\limits_{i \in \mathcal X} \sum_{i' \in \mathcal X: V(i) < V(i')} q_{ii'}.$$
\end{lemma}

Based on drift analysis in Lemma \ref{DriftBound} and Lemma \ref{TailBound} (Lemma B.1 in \cite{WanMagSri_17}), we have the following tail bound on $V(S)$.

\begin{lemma} \label{lem:SSC}
Given the Lyaponuv function defined in \eqref{ly-func} and denote $\tilde{k} = 1+\frac{1}{4(b-1)}$, we have $$\Pr\left(V(S)\geq \frac{\tilde{k}\log N}{\sqrt{N}}\right)\leq e^{-\frac{\log^2 N}{32(b-1)^2}+\frac{\log N}{16(b-1)}}.$$
\end{lemma}

\begin{proof}
From Lemma \ref{DriftBound}, we have
\begin{align*}
B=\frac{\log N}{\sqrt{N}} ~\text{and}~ \gamma=\frac{1}{2(b-1)}\frac{\log N}{\sqrt{N}}-\frac{1}{\sqrt{N}},
\end{align*}
and it is easy to verify
\begin{align*}
q_{\max}\leq N ~\text{and}~ v_{\max}\leq \frac{1}{N}.
\end{align*}

Based on Lemma \ref{TailBound}, we have for sufficiently large $N,$
\begin{align*}
\Pr\left(V(S)\geq \frac{\tilde{k}\log N}{\sqrt{N}}\right) \leq &\left(\frac{1}{1+\frac{1}{2(b-1)}\frac{\log N}{\sqrt{N}}-\frac{1}{\sqrt{N}}}\right)^{\frac{\sqrt{N}\log N}{8(b-1)}+1}\\
\leq& \left(1 - \frac{1}{4(b-1)}\frac{\log N}{\sqrt{N}}+\frac{1}{2\sqrt{N}}\right)^{\frac{\sqrt{N}\log N}{8(b-1)}+1}\\ \leq& e^{-\frac{\log^2 N}{32(b-1)^2}+\frac{\log N}{16(b-1)}}.
\end{align*}

\end{proof}

Given the gradient bounds in Lemma \ref{lemma:g'} and Lemma \ref{lemma:g''}, and SSC in Lemma \ref{lem:SSC}, we are ready to bound all the terms in \eqref{G-expansion} and establish the main theorem. Considering the first term in \eqref{G-expansion}, we have
\begin{align*}
&E\left[g'\left(\sum_{i=1}^{b} S_i\right)\left(\lambda B(S)-\lambda-\frac{\log N}{\sqrt{N}}+S_1\right)\mathbb{I}_{\sum_{i=1}^{b} S_i> \lambda +\frac{k\log N}{\sqrt{N}}+\frac{1}{N}}\right]\\
\leq&E\left[\frac{\sqrt{N}}{\log N}\left(\sum_{i=1}^{b} S_i-\lambda -\frac{k\log N}{\sqrt{N}}\right)\left(\lambda+\frac{\log N}{\sqrt{N}}-S_1\right)\mathbb{I}_{\sum_{i=1}^{b} S_i> \lambda +\frac{k\log N}{\sqrt{N}}+\frac{1}{N}}\right]\\
=&E\left[\frac{\sqrt{N}}{\log N}\left(\sum_{i=1}^{b} S_i-\lambda -\frac{k\log N}{\sqrt{N}}\right)\left(\lambda+\frac{\log N}{\sqrt{N}}-S_1\right)\mathbb{I}_{V(S)\leq \frac{\tilde{k}\log N}{\sqrt{N}}}\mathbb{I}_{\sum_{i=1}^{b} S_i> \lambda +\frac{k\log N}{\sqrt{N}}+\frac{1}{N}}\right]\\
&+E\left[\frac{\sqrt{N}}{\log N}\left(\sum_{i=1}^{b} S_i-\lambda -\frac{k\log N}{\sqrt{N}}\right)\left(\lambda+\frac{\log N}{\sqrt{N}}-S_1\right)\mathbb{I}_{V(S)>\frac{\tilde{k}\log N}{\sqrt{N}}}\mathbb{I}_{\sum_{i=1}^{b} S_i> \lambda +\frac{k\log N}{\sqrt{N}}+\frac{1}{N}}\right].
\end{align*}
Note that given $\sum_{i=1}^{b} S_i> \lambda +\frac{k\log N}{\sqrt{N}}+\frac{1}{N},$ $$V(S)=\lambda+\frac{k\log N}{\sqrt{N}}-S_1.$$ Hence,
$V(S)\leq \frac{\tilde{k} \log N}{\sqrt{N}}$ implies that 
\begin{align*}
\lambda+\frac{\log N}{\sqrt{N}}-S_1 \leq& \left(\tilde{k} - k + 1 \right) \frac{\log N}{\sqrt{N}}\\
 \leq& \left(1 - \frac{1}{4(b-1)} \right) \frac{\log N}{\sqrt{N}}.
\end{align*}
Therefore, we have
\begin{align*}
&E\left[g'\left(\sum_{i=1}^{b} S_i\right)\left(\lambda B(S)-\lambda-\frac{\log N}{\sqrt{N}}+S_1\right)\mathbb{I}_{\sum_{i=1}^{b} S_i> \lambda +\frac{k\log N}{\sqrt{N}}+\frac{1}{N}}\right]\\
\leq& \left(1 - \frac{1}{4(b-1)} \right)E\left[\max\left\{\sum_{i=1}^{b} S_i-\lambda -\frac{k\log N}{\sqrt{N}},0\right\}\right]+\frac{(b+1)\sqrt{N}}{\log N} e^{-\frac{\log^2 N}{32(b-1)^2}+\frac{\log N}{16(b-1)}}.
\end{align*}

For the second and third terms in \eqref{G-expansion}, based on Lemma \ref{lemma:g'} and Lemma \ref{lemma:g''}, we have
\begin{align*}
&\frac{2}{N}\max_{\sum_{i=1}^{b} S_i > \lambda +\frac{k\log N}{\sqrt{N}}}\left|g''\left(\sum_{i=1}^{b} S_i\right)\right| \leq \frac{2}{\sqrt{N}\log N}, \\
&E\left[\left(g'\left(\sum_{i=1}^{b} S_i\right)\left(-\frac{\log N}{\sqrt{N}}\right)+|\lambda-\lambda B(S)|g'(\xi)+|S_1|g'(\tilde{\xi})\right)\mathbb{I}_{-\frac{1}{N}\leq \sum_{i=1}^{b} S_i-\lambda -\frac{k\log N}{\sqrt{N}}\leq \frac{1}{N}}\right] \\
&\leq \frac{5}{\sqrt{N}\log N}.
\end{align*}

In summary, we have for sufficiently large $N,$
\begin{align*}
&E\left[\max\left\{\sum_{i=1}^{b} S_i-\lambda -\frac{k\log N}{\sqrt{N}},0\right\}\right]\\
\leq & \left(1 - \frac{1}{4(b-1)}\right) E\left[\max\left\{\sum_{i=1}^{b} S_i-\lambda -\frac{k\log N}{\sqrt{N}},0\right\}\right]+\frac{(b+1)\sqrt{N}}{\log N} e^{-\frac{\log^2 N}{32(b-1)^2}+\frac{\log N}{16(b-1)}}+\frac{7}{\sqrt{N}\log N},
\end{align*} which implies that
\begin{align*}
E\left[\max\left\{\sum_{i=1}^{b} S_i-\lambda -\frac{k\log N}{\sqrt{N}},0\right\}\right]\leq \frac{29(b-1)}{\sqrt{N}\log N}
\end{align*} when $N$ is sufficiently large.

\section{Proof of the Corollary} \label{sec:0-delay}
Under JSQ, a job is discarded or blocked only if all buffers are full, i.e. when $\sum_{i=1}^b S_i=b.$  From Theorem \ref{Thm:main},
\begin{align*}
p_{\mathcal B_N} = &\Pr\left(\sum_{i=1}^b S_i=b\right)=\Pr\left(\sum_{i=1}^b S_i\geq b\right)\\
\leq &\Pr\left(\max\left\{\sum_{i=1}^b S_i-\lambda-\frac{k\log N}{\sqrt{N}}, 0\right\}\geq b-\lambda-\frac{k\log N}{\sqrt{N}}\right)\\
\stackrel{(a)}{\leq}&\frac{E\left[\max\left\{\sum_{i=1}^b S_i-\lambda-\frac{k\log N}{\sqrt{N}}, 0\right\}\right]}{b-\lambda-\frac{k\log N}{\sqrt{N}}}\\
\leq&\frac{29}{\sqrt{N}\log N}.
\end{align*}
where $(a)$ holds due to the Markov inequality. For jobs that are not discarded, the average queueing delay according to Little's law is
$$\frac{E\left[\sum_{i=1}^bS_i\right]}{\lambda(1-p_{\mathcal B_N})}.$$ Therefore, the average waiting time is
$$E[W_N]=\frac{E\left[\sum_{i=1}^bS_i\right]}{\lambda(1-p_{\mathcal B_N})}-1\leq \frac{\lambda+\frac{k\log N}{\sqrt{N}}+\frac{29(b-1)}{\sqrt{N}\log N}}{\lambda(1-p_{\mathcal B_N})}-1=\frac{\frac{k\log N}{\sqrt{N}}+\frac{29(b-1)}{\sqrt{N}\log N}+\lambda p_{\mathcal B_N}}{\lambda(1-p_{\mathcal B_N})}\leq \frac{3\log N}{\sqrt{N}}.$$ Finally, a job not routed to an idle server is either blocked or waited in the buffer
$$p_{\mathcal W_N}=  p_{\mathcal B_N}+\frac{E[W_N]}{1-p_{\mathcal B_N}},$$ where last term is from the fact that
$$\lambda \Pr\left(\hbox{a job is routed to a busy server with empty buffer}\right)=E[W_N]\lambda(1-p_{\mathcal B_N}).$$ Therefore, we have
$$p_{\mathcal W_N}\leq \frac{4\log N}{\sqrt{N}}.$$

The exact same analysis holds for I1F since it gives priority to idle servers and servers with only one job. Analysis for Po$d$ is similar, except that
\begin{align*}
p_{\mathcal B_N}=&\Pr\left(\mathcal B_N\left|S_b\leq 1-\frac{1}{N^{\alpha}}\right.\right)\Pr\left(S_b\leq 1-\frac{1}{N^{\alpha}}\right)\\
&+\Pr\left(\mathcal B_N\left|S_b> 1-\frac{1}{N^{\alpha}}\right.\right)\Pr\left(S_b> 1-\frac{1}{N^{\alpha}}\right)\\
\leq&\Pr\left(\mathcal B_N\left|S_b\leq 1-\frac{1}{N^{\alpha}}\right.\right)+\Pr\left(S_b> 1-\frac{1}{N^{\alpha}}\right)\\
\leq&\left(1-\frac{1}{N^{\alpha}}\right)^{N^\alpha \log N}+\Pr\left(\sum_{i=1}^b S_i>b-\frac{b}{N^{\alpha}}\right)\\
\leq&\frac{30}{\sqrt{N}\log N}.
\end{align*} The remaining analysis is the same.

Finally, for JIQ, we have not been able to bound $p_{\mathcal B_N}.$ However,
$$p_{\mathcal W_N}=\Pr\left(S_1=1\right)\leq \Pr\left(\sum_{i=1}^bS_i\geq 1\right)\leq \Pr\left(\max\left\{\sum_{i=1}^bS_i-\lambda-\frac{k\log N}{\sqrt{N}}\right\}\geq \frac{1}{N^{\alpha}}-\frac{k\log N}{\sqrt{N}}\right).$$ The result follows from the Markov inequality.

\section{conclusion}

In this paper, we studied the steady-state performance of a class of load balancing algorithms for many-server ($N$ servers) systems in the sub-Halfin-Whitt regime. We established an upper bound on the expected queue length with Stein's method and studied the probability that an incoming job is routed to a busy server under JSQ, I1F, JIQ, and Po$d$.


\section*{Acknowledgement} The authors would like to thank Anton Braverman and Weina Wang for stimulating discussions that led to this result.

\bibliographystyle{abbrv}
\bibliography{U:/bib/inlab-refs}

\end{document}